\pdfoutput=1
\RequirePackage{ifpdf}
\ifpdf % We are running pdfTeX in pdf mode
\documentclass[pdftex]{sigma}
\else
\documentclass{sigma}
\fi

\usepackage{etoolbox,enumitem,tikz}
\tikzset{>=latex}

\newtheorem{prop}{Proposition}
\newtheorem{lem}[prop]{Lemma}
\newtheorem{thm}[prop]{Theorem}

\theoremstyle{definition}
\newtheorem{df}[prop]{Definition}
\newtheorem{rem}[prop]{Remark}
\newtheorem{ex}[prop]{Example}

\newcommand{\R}{\mathbb{R}}
\newcommand{\C}{\mathbb{C}}

\newcommand{\id}{{\rm Id}}
\newcommand{\h}{\mbox{\ensuremath{[\hspace{-0.8pt}[\hslash]\hspace{-0.8pt}]}}}
\newcommand{\CP}{\mathbb{CP}}

\newcommand{\ideal}[1]{\left(#1\right)}
\newcommand{\inner}[1]{\left<\smash[t]{#1}\right>}

\newcommand{\mat}[1]{\left(\begin{matrix} #1 \end{matrix}\right)}

\begin{document}
\allowdisplaybreaks

\newcommand{\arXivNumber}{1912.01225}

\renewcommand{\PaperNumber}{050}

\FirstPageHeading

\ShortArticleName{On the Notion of Noncommutative Submanifold}
\ArticleName{On the Notion of Noncommutative Submanifold}

\Author{Francesco D'ANDREA}
\AuthorNameForHeading{F.~D'Andrea}

\Address{Universit\`a di Napoli ``Federico II'' and I.N.F.N. Sezione di Napoli, Complesso MSA,\\ Via Cintia, 80126 Napoli, Italy}
\Email{\href{mailto:francesco.dandrea@unina.it}{francesco.dandrea@unina.it}}
\URLaddress{\url{http://wpage.unina.it/francesco.dandrea/}}

\ArticleDates{Received January 11, 2020, in final form May 30, 2020; Published online June 09, 2020}

\Abstract{We review the notion of submanifold algebra, as introduced by T.~Masson, and discuss some properties and examples. A submanifold algebra of an associative algebra $A$ is a quotient algebra $B$ such that all derivations of $B$ can be lifted to~$A$. We will argue that in the case of smooth functions on manifolds every quotient algebra is a submanifold algebra, derive a topological obstruction when the algebras are deformation quantizations of symplectic manifolds, present some (commutative and noncommutative) examples and counterexamples.}

\Keywords{submanifold algebras; tangential star products; coisotropic reduction}

\Classification{46L87; 53C99; 53D55; 13N15}

\section{Introduction}
Inspired by Gelfand duality, establishing that any commutative $C^*$-algebra is isomorphic to the algebra of continuous functions vanishing at infinity on a locally compact Hausdorff space, the point of view of noncommutative geometry is to regard any associative algebra (possibily with additional structure, e.g., a~Dirac operator~\cite{Con}) as describing some virtual ``noncommutative'' space. It is natural to wonder what is the correct notion of noncommutative space contained within another noncommutative space.

Motivated by the derivation-based differential calculus of M.~Dubois-Violette and P.W.~Michor \cite{DV,DM} (reviewed, e.g., in~\cite{DV2,DV3}),
T.~Masson introduced in \cite{Mas} the notion of submanifold algebra as a possible candidate for what in the noncommutative realm should be the analogue of a closed embedded submanifold of a smooth manifold (see \cite{AN} for another possible approach).
The starting point is a short exact sequence of associative algebras (over a field $K$):
\begin{gather}\label{eq:shortexact}
0\to J\to A\xrightarrow{\pi}B\to 0.
\end{gather}
The sequence \eqref{eq:shortexact} induces an exact sequence of complexes
\begin{gather}\label{eq:hochcomplex}
0\to\operatorname{Hom}\big(A^{\otimes n},J\big)\to\operatorname{Hom}_\pi\big(A^{\otimes n},A\big)\to\operatorname{Hom}\big(B^{\otimes n},B\big),
\end{gather}
where the first one is the Hochschild complex of $A$ with coefficients in the $A$-bimodule $J$, the last one is the Hochschild complex of $B$, and the one in the middle is the sub-complex of the Hochschild complex of $A$ given by $K$-linear maps $A^{\otimes n}\to A$ with image in~$J$ if one of the arguments is in~$J$.

Last arrow in \eqref{eq:hochcomplex} is, in general, not surjective. It is for $n=0$, since \eqref{eq:hochcomplex} reduces to \eqref{eq:shortexact}, and for $n=1$ if restricted to Hochschild coboundaries, i.e., inner derivations. The restriction to Hochschild $1$-cocycles, i.e., derivations, already gives a map{\samepage
\begin{gather}\label{eq:piDer}
\pi_*\colon \ \operatorname{Der}_\pi(A)\to\operatorname{Der}(B)
\end{gather}
that may not be surjective (cf.~Examples \ref{ex:no1}, \ref{ex:cross2} and \ref{ex:no2}).}

If $M$ is a smooth manifold and $S\subset M$ a closed embedded smooth submanifold, it is well known that the pullback of the inclusion $\imath\colon S\to M$ is surjective, and one has an exact sequence
\begin{gather*}
0\to J\to C^\infty(M)\xrightarrow{\imath^*}C^\infty(S)\to 0
\end{gather*}
like in \eqref{eq:shortexact}, with $J$ the ideal of smooth functions on $M$ that are zero on~$S$. The induced map~\eqref{eq:piDer} on vector fields (derivations) is surjective as well.

With this example in mind, whenever we have a sequence~\eqref{eq:shortexact} such that \eqref{eq:piDer} is surjective, we will call $B$ a \emph{submanifold algebra} of~$A$~\cite{Mas}.

Submanifold algebras have been recently studied from the point of view of Drinfel'd twists in \cite{FFW,FW,Web}.

The aim of this paper is to understand the meaning of submanifold algebras from the point of view of noncommutative geometry.
The plan is the following. In Section~\ref{sec:pre} we fix the notations and provide some algebraic background.
In Section~\ref{sec:CinftyM} we investigate the case of smooth manifolds (and closed subsets), exhibit examples of commutative submanifold algebras that are not algebras of functions on smooth manifolds, examples of quotient algebras that are not submanifold algebras, and argue that when both $A$ and $B$ in \eqref{eq:shortexact} are algebras of smooth functions the condition on $\pi_*$ is in fact redundant (Theorem \ref{thm:main}).
In Section~\ref{sec:ncexamples} we discuss some examples of submanifold algebras of noncommutative algebras, including ``almost commutative'' spaces. In Section~\ref{sec:starproducts} we study deformation quantizations of Poisson manifolds and discuss some topological obstructions for a quotient algebra to be a submanifold algebra (cf.~Lemma~\ref{lemma:obstr}).

\section{Preliminaries and notations}\label{sec:pre}
\subsection{Notations}
Throughout the following, and unless stated otherwise, $A$ will denote an associative algebra over a field $K$, $M$ a smooth manifolds (without boundary), $C^\infty(M)$ the algebra of \emph{real}-valued smooth function on $M$, $C_0(M)$ the $C^*$-algebra of \emph{complex}-valued continuous functions vanishing at infinity. For $S\subset M$ we define
\begin{gather*}
I(S):=\big\{f\in C_0(M)\colon f|_S=0\big\}
\end{gather*}
and call both $I(S)$ and $I(S)\cap C^\infty(M)$ the \emph{vanishing ideal} of $S$ (it will be clear from the context which one we are referring to). Recall that a \emph{derivation} $D$ of an algebra $A$ is a $K$-linear map $A\to A$ satisfying the Leibniz rule
\begin{gather}\label{eq:Leibniz}
D(ab)=D(a)b+aD(b),\qquad a,b\in A.
\end{gather}
The set of all derivations of $A$ will be denoted by $\operatorname{Der}(A)$. A derivation is \emph{inner} if it is of the form $a\mapsto [x,a]$ for some fixed $x\in A$. The set of all inner derivations will be denoted by $\operatorname{Inn}(A)$.

\subsection{Submanifold algebras}
Let $\pi\colon A\to B$ be a surjective homorphism of associative algebras over a field $K$ and $J:=\ker\pi$.
Let
\begin{gather*}
\operatorname{Der}_\pi(A):=\big\{D\in\operatorname{Der}(A)\colon D(a)\in J \ \forall\, a\in J\big\}.
\end{gather*}

\begin{lem}For every $D\in\operatorname{Der}_\pi(A)$ there is a unique derivation $\widetilde{D}\in\operatorname{Der}(B)$ such that $\widetilde{D}\circ\pi=\pi\circ D$.
\end{lem}

\begin{proof}
Let $a,b\in A$ with $a-b\in J$. Then $D(a)-D(b)\in J$ and $\pi(D(a))\in B$ depends only on the class $x:=\pi(a)$ of $a$. We define $\widetilde{D}(x):=\pi(D(a))$. Now if $x=\pi(a)$ and $y=\pi(b)$, then $xy=\pi(ab)$, and
\begin{gather*}
\widetilde{D}(xy)=\pi(D(ab))=\pi(D(a)b+aD(b))=\widetilde{D}(x)y+x\widetilde{D}(y).
\end{gather*}
Similarly one proves that $\widetilde{D}$ is linear.
\end{proof}

The assignment $D\mapsto \widetilde{D}$ gives a Lie algebra map:
\begin{gather*}
\pi_*\colon \ \operatorname{Der}_\pi(A)\to\operatorname{Der}(B).
\end{gather*}

\begin{df}\label{def:sMA}
If $\pi_*$ is surjective, we will call $\pi\colon A\to B$ a \emph{coembedding}\footnote{Since in the motivating example of smooth manifolds it is somehow dual to a~(closed) embedding.} and $B$ a \emph{submanifold algebra} of $A$.
\end{df}

Surjectivity of \eqref{eq:piDer} guarantees, among other things, that there is a surjective homomorphism $\Omega_{\operatorname{Der}}(A)\to\Omega_{\operatorname{Der}}(B)$ of ``minimal'' derivation based differential calculi \cite[Proposition~IV.1]{Mas}.

It is straightforward to check that, since $J$ is a two-sided ideal, $\operatorname{Inn}(A)$ is a Lie subalgebra of $\operatorname{Der}_\pi(A)$ and $\pi_*$ maps $\operatorname{Inn}(A)$ \emph{surjectively} to $\operatorname{Inn}(B)$.

\begin{rem}\label{rem:inner}If all derivations of $B$ are inner, then $B$ is a submanifold algebra of~$A$.
\end{rem}

In Section~\ref{sec:CinftyM} we will study the case of commutative algebras, where no non-zero derivation is inner.

\subsection{Compositions of coembeddings}

Given a sequence \mbox{$A\xrightarrow{f}B\xrightarrow{g}C$} of coembeddings, in general $g_*\circ f_*$ and $(g\circ f)_*$ are not equal (they don't even have the same domain). It is then not obvious that the composition $g\circ f$ is still a coembedding. In order to discuss compositions of coembeddings, it is
convenient to rephrase properties of derivations in terms of homomorphisms of associative algebras.
We are going to need the next lemma, whose proof is straightforward.

\begin{lem}\label{lemma:17}
Let $A$ be an associative algebra. A map $D\colon A\to A$ is a derivation if and only if the map
\begin{gather}\label{eq:admissible}
A\to M_2(A),\qquad a\mapsto\mat{a & D(a) \\ 0 & a}
\end{gather}
is a homomorphism $($of associative algebras$)$.
\end{lem}

A homomorphism $A\to M_2(A)$ will be called \emph{admissible} if it is of the form~\eqref{eq:admissible}
for some (unique) $D\in\operatorname{Der}(A)$.

Given a map $f\colon A\to B$, we will denote by $F\colon M_2(A)\to M_2(B)$ the map obtained by app\-lying~$f$ to each matrix element. The definition of coembedding can then be restated in the following form.

\begin{lem}A surjective homomorphism $f\colon A\to B$ is a coembedding if and only if for every admissible homomorphism $\xi\colon B\to M_2(B)$ there exists an admissible homomorphism $\eta\colon A\to M_2(A)$ making the following diagram commute
\begin{center}
\begin{tikzpicture}

\node (A) at (0,0) {$A$};
\node (B) at (2,0) {$B$};
\node (C) at (0,1.5) {$M_2(A)$};
\node (D) at (2,1.5) {$M_2(B)$};

\draw (A) edge[->] node[above] {\footnotesize $f$} (B)
 (C) edge[->] node[above] {\footnotesize $F$} (D)
			(A) edge[->] node[left] {\footnotesize $\eta$} (C)
			(B) edge[->] node[right] {\footnotesize $\xi$} (D);

\end{tikzpicture}
\end{center}
\end{lem}

It is now easy to prove the following:

\begin{prop}Consider a sequence \mbox{$A\xrightarrow{f}B\xrightarrow{g}C$} of maps between associative algebras. If $f$ and $g$ are coembeddings, then $g\circ f$ is a coembedding.
\end{prop}

\begin{proof}The composition $g\circ f$ of two surjective homomorphisms is a surjective homomorphisms. Consider the diagram:
\begin{center}
\begin{tikzpicture}

\node (A) at (0,0) {$A$};
\node (B) at (2.5,0) {$B$};
\node (C) at (0,1.8) {$M_2(A)$};
\node (D) at (2.5,1.8) {$M_2(B)$};
\node (E) at (5,0) {$C$};
\node (F) at (5,1.8) {$M_2(C)$};

\draw (A) edge[->] node[above] {\footnotesize $f$} (B)
 (B) edge[->] node[above] {\footnotesize $g$} (E)
 (C) edge[->] node[above] {\footnotesize $F$} (D)
 (D) edge[->] node[above] {\footnotesize $G$} (F)
			(A) edge[dashed,->] node[left] {\footnotesize $\eta$} (C)
			(B) edge[dashed,->] node[fill=white] {\footnotesize $\xi$} (D)
			(E) edge[->] node[right] {\footnotesize $\phi$} (F);

\end{tikzpicture}
\end{center}
If $\phi$ is admissible, since $g$ is a coembedding there exists an admissible $\xi$ making the right square commute. But $f$ is a coembedding as well, so there exists an admissible $\eta$ making the left square~-- and then the outer rectangle~-- commute. Since for every $\phi$ there exists $\eta$ making the outer rectangle commute, $g\circ f$ is a coembedding.
\end{proof}

\section{Function algebras}\label{sec:CinftyM}

\subsection[Submanifold algebras of commutative $C^*$-algebras]{Submanifold algebras of commutative $\boldsymbol{C^*}$-algebras}

Let us start with topological spaces. Suppose that \eqref{eq:shortexact} is an exact sequence of \emph{commutative} complex $C^*$-algebras. It is well known that, by Gelfand duality, $A\simeq C_0(M)$ and $B\simeq C_0(N)$
where $M$ and $N$ are locally compact Hausdorff spaces. Every $*$-homomorphism $\pi\colon C_0(M)\to C_0(N)$ is the pullback of a proper continuous map $F\colon N\to M$, and it is surjective only if~$F$ is injective: it is then a topological embedding of $N$ as a closed subset of $M$. Up to isomorphisms, every exact sequence~\eqref{eq:shortexact} of commutative $C^*$-algebras is then of the form
\begin{gather*}
0\to I(S)\to C_0(M)\xrightarrow{\pi} C_0(S)\to 0
\end{gather*}
with $S\subset M$ a closed topological subspace and
$\pi(f):=f|_S \ \forall\, f\in C_0(M)$
(and every inclusion $S\hookrightarrow M$ of a closed subset gives rise to such a short exact sequence).
We refer to Chapter~1 of~\cite{GVF} for the details.

Since commutative $C^*$-algebras have no non-zero derivations,\footnote{In fact, a stronger statement is due to Kadison \cite[Theorem~2]{Kad}: each derivation of a (possibly non-commutative) $C^*$-algebra annihilates its center.} the associated map $\pi_*$ is surjective as well.

\subsection[Submanifold algebras of $C^\infty$-algebras]{Submanifold algebras of $\boldsymbol{C^\infty}$-algebras}
The crucial point in the discussion in previous section is that the functor $C_0$ is an equivalence between the category of locally compact Hausdorff topological spaces, with proper continuous functions, and the (opposite) category of commutative $C^*$-algebras with $*$-homomorphisms (see, e.g., \cite[Section~II.2.2.7]{Bla}).

Things with smooth manifolds become more involved. We can associate to every smooth manifold (without boundary) $M$ the algebra $C^\infty(M)$, and to any smooth map $F\colon M\to N$ the homomorphism $F^*\colon C^\infty(N)\to C^\infty(M)$. This gives a functor from the category of smooth manifolds to the category of commutative (real unital associative) algebras, which obviously fails to be surjective on objects.\footnote{We could, however, consider the category of $C^\infty$-rings, for which there are characterizations of those objects that are isomorphic to $C^\infty(M)$ for some $M$. See, e.g.,~\cite{MV}.} Nevertheless it is a full and faithful functor (see \cite[Corollary~35.10]{KMS} or \cite[Theorem~2.8]{MR}), and this is enough for our purposes.

Let us consider then a short exact sequence
\begin{gather}\label{eq:exactcommutative}
0\to J\to C^\infty(M)\xrightarrow{\pi}B\to 0,
\end{gather}
where we know a priori that the one in the middle is the algebra of smooth functions on some smooth manifold $M$. We will assume that $\dim(M)\geq 1$ and use the standard identification of $\operatorname{Der}(C^\infty(M))$ with the set $\mathfrak{X}(M)$ of smooth global vector fields on~$M$.

There are several natural questions we may ask:
\begin{enumerate}[label=(\roman*)]\itemsep=0pt
\item Is there a commutative example where $\pi$ is surjective but $\pi_*$ is not? (Are these two conditions independent from each other or is every quotient algebra of $C^\infty(M)$ a~submanifold algebra?) \label{quest:1}

\item Is there an example where $\pi$ and $\pi_*$ are both surjective ($B$ is a submanifold algebra of~$C^\infty(M)$), but $B\not\simeq C^\infty(N)$ for any smooth manifold~$N$?\label{quest:2}

\item If we known that $J$ is the vanishing ideal of a subset $S\subset M$ and $B\simeq C^\infty(M)/J$ a~submanifold algebra, can we conclude that $S$ is a submanifold of~$M$? \label{quest:3}

\item What can we conclude under the assumption that $B=C^\infty(N)$? \label{quest:4}

\end{enumerate}
The first two examples give a positive answer to question \ref{quest:2}.

\begin{ex}[dual numbers]
Let $M=\R$ and $B$ be the subalgebra of $M_2(\R)$ spanned by the identity matrix and the matrix
\begin{gather*}
\varepsilon:=\mat{0 & 1 \\ 0 & 0}.
\end{gather*}
Up to an isomorphism, $B\simeq\R[\varepsilon]/\ideal{\varepsilon^2}$ is the algebra of dual numbers.
Let $\pi\colon C^\infty(\R)\to B$ be the homomorphism (the identity matrix is omitted):
\begin{gather*}
\pi(f):=f(0)+f'(0)\varepsilon \qquad \forall\, f\in C^\infty(\R).
\end{gather*}
Clearly $\pi$ is surjective, and it is not difficult to check that $\pi_*$ is surjective as well.
One has
\begin{gather*}
\operatorname{Der}(C^\infty(\R)) =\big\{f\mapsto \phi f'\colon \phi\in C^\infty(\R)\big\}, \\
\operatorname{Der}(B)  =\left\{\lambda\varepsilon\tfrac{{\rm d}}{{\rm d}\varepsilon} \colon \lambda\in\R\right\}.
\end{gather*}
Here $J=\ker\pi$ is given by those smooth functions on $\R$ vanishing at $0$ together with their first derivative.
Given a derivation $D\colon f\mapsto\phi f'$, from $(\phi f')'(0)=\phi(0)f''(0)+\phi'(0)f'(0)$ it follows that the derivation maps $J$ to $J$ if and only if $\phi(0)=0$. Hence:
\begin{gather*}
\operatorname{Der}_\pi(C^\infty(\R)) =\big\{f\mapsto \phi f'\colon \phi\in C^\infty(\R),\,\phi(0)=0\big\}.
\end{gather*}
If $Df:=\phi f'$ with $\phi(0)=0$, then $\pi_*(D)=\phi'(0)\varepsilon\frac{{\rm d}}{{\rm d}\varepsilon}$. Since $\phi'(0)$ can be any real number, the map $\pi_*$ is surjective.
\end{ex}

We saw in previous example that: the algebra of dual numbers is a submanifold algebra of~$C^\infty(\R)$; it is not isomorphic to any algebra of smooth functions on a~manifold, since it has a~non-zero nilpotent element~$\varepsilon$; the kernel of $\pi$ is not the vanishing ideal of any subset of $\R$, even if $\operatorname{Der}_\pi(C^\infty(\R))$ is the set of vector fields on $\R$ vanishing on the subset $S=\{0\}$.

More generally, for arbitrary $M$, given a point $p\in M$ and a non-zero tangent vector $v\in T_pM$, we can construct a surjective homomorphism
\begin{gather*}
\pi\colon \ C^\infty(M)\to \R[\varepsilon]/\ideal{\varepsilon^2},\qquad
\pi(f):=f(p)+v(f)\varepsilon,
\end{gather*}
and prove (using local coordinates) that $\pi_*$ is surjective as well.

\begin{ex}[germs of smooth functions]
Fix $p\in M$, let $B:=C_p^\infty(M)$ be the algebra of all germs of smooth functions at $p$, and let
\begin{gather*}
\pi\colon \ C^\infty(M)\to C_p^\infty(M)
\end{gather*}
be the map sending a function to its germ at $p$. Such a map is surjective (every germ can be represented by a global smooth function). Its kernel $J$ is the set of all functions that are zero \emph{near} $p$;
since derivations are local, $X(J)\subset J$ for every $X\in\operatorname{Der}(C^\infty(M))$ and $\operatorname{Der}_\pi(C^\infty(M))\equiv \operatorname{Der}(C^\infty(M))$.
The map
\begin{gather*}
\pi_*\colon \ \mathfrak{X}(M)\to T_pM
\end{gather*}
is just the evaluation of vector fields at $p$. It is well known that such a map is surjective \cite[Proposition~8.7]{Lee}.
Stalks of germs of smooth functions on $M$ are then examples of submanifold algebras of $C^\infty(M)$ not isomorphic to $C^\infty(N)$ for any smooth manifold~$N$.\footnote{By contraddiction, assume that $C_p^\infty(M)\simeq C^\infty(N)$ for some smooth manifold $N$. Since $\mathfrak{X}(N)\simeq T_pM$ is finite-dimensional, $N$ must be $0$-dimensional. But then $\mathfrak{X}(N)$ must be $\{0\}$, that implies $\dim(M)=\dim(T_pM)=0$.}
\end{ex}

We want now to answer to question~\ref{quest:3}. If $S\subset M$ is a closed subset of a smooth manifold $M$, we denote by $C^\infty(S)^{\circ}$ the algebra of real-valued functions on $S$ that admit a smooth prolongation to an open neighbourhood of $S$ in $M$ (and then to the whole $M$, cf.~\cite[Lemma~2.26]{Lee}). This is the standard way of defining smooth functions on a general subset of a manifold, when such a~subset has no intrinsic smooth structure. If~$S$ is a closed embedded submanifold, then~$C^\infty(S)^{\circ}$ coincides with the set of functions that are smooth with respect to the smooth structure of~$S$. (However, it should be noted that if~$S$ is an immersed submanifold this is not true.)

\begin{thm}[closed subsets]\label{thm:closed}
$C^\infty(S)^\circ$ is a submanifold algebra of $C^\infty(M)$.
\end{thm}

\begin{proof}By construction, elements of $C^\infty(S)^\circ$ are functions on $S$ that are in the image of the restriction map $\pi\colon C^\infty(M)\to C^\infty(S)^\circ$, $\pi(f):=f|_S$. We need to prove surjectivity of $\pi_*$.

Any derivation of $C^\infty(S)^\circ$ composed with the evaluation at a point $p\in S$ gives a derivation of $C^\infty(M)$ at $p$, i.e., a vector $X_p\in T_pM$. We need to prove that the map $X\colon S\to TM$, $p\mapsto X_p$, is the restriction to $S$ of a global smooth vector field on~$M$.

Let $p\in S$, let $\big(U,\varphi=\big(x^1,\ldots,x^n\big)\big)$ be a chart on $M$ centered at $p$ and $B$ a coordinate ball with $p\in B$ and $\overline{B}\subset U$. For all $q\in U\cap S$ we can write
\begin{gather*}
X_q=\sum_{i=1}^nv^i(q)\frac{\partial}{\partial x^i}\Big|_q
\end{gather*}
for some (unique) $v^i(q)\in\R$.
Choose smooth functions $\widetilde{x}^i\in C^\infty(M)$ such that $\widetilde{x}^i$ coincide with~$x^i$ on $\overline{B}$ \cite[Lemma~2.26]{Lee}. For every $i=1,\ldots,n$, since $X(\widetilde{x}^i|_S)\in C^\infty(S)^\circ$ and~$\pi$ is surjective, we can choose $f^i\in C^\infty(M)$ such that
\begin{gather*}
X\big(\widetilde{x}^i|_S\big)=\pi\big(f^i\big).
\end{gather*}
For all $q\in\overline{B}\cap S$, one has $f^i(q)=X_q\big(\widetilde{x}^i\big)=X_q(x^i)=v^i(q)$.
As a consequence, the smooth vector field $Y\in\mathfrak{X}(U)$ given by
\begin{gather*}
Y:=f^i\frac{\partial}{\partial x^i}
\end{gather*}
satisfies $Y|_B=X|_B$. The local vector field $X$ is locally the restriction of a local smooth vector field on $M$ (it is a smooth vector field \emph{along} $S$ according to the terminology of \cite{Lee}). It follows from \cite[Lemma 8.6]{Lee} that $X$ is the restriction to $S$ of a global smooth vector field on $M$.
\end{proof}

Thus, for \emph{any} closed subset $S$ of a smooth manifold $M$, $C^\infty(S)^\circ$ is a submanifold algebra of $C^\infty(M)$ (with $\pi$ the pullback of the inclusion $S\hookrightarrow M$ and $J$ the vanishing ideal of $S$).
This notion of noncommutative submanifold, even in the commutative case, covers then more examples than the standard geometric notion of submanifold. This is in some sense analogous to what happens in diffeology, where on any subset of a diffeological space there is a canonical ``subset diffeology'' induced by the ambient space \cite{IZ}.

We can now answer to question \ref{quest:1} and give a commutative example of quotient algebra that is \emph{not} a submanifold algebra.

\begin{ex}[the cross]\label{ex:no1}
Let $S:=\big\{(x,y)\in\R^2\colon xy=0\big\}$,
let $\imath\colon \R\to S$ be the map $x\mapsto (x,0)$,
and $\pi=\imath^*\colon C^\infty(S)^\circ\to C^\infty(\R)$ its pullback.
It follows from Theorem~\ref{thm:closed} that a derivation $X$ of $C^\infty(S)^\circ$ is a smooth vector field along~$S$ of the form
\begin{gather*}
X_p=\sum_{i=1}^2v^i(p)\frac{\partial}{\partial x^i}\Big|_p,
\end{gather*}
where $x^1$, $x^2$ are the standard coordinates on $\R^2$.
Let $f(x,y):=xy$. The condition that $X_p(f)=0$ $\forall\, p\in S$
implies that $v^1(0,y)=0$ $\forall\, y\neq 0$ and $v^2(x,0)=0$ $\forall\, x\neq 0$. By continuity, $X_{p=0}=0$. The image $\pi_*(X)$ is a vector field on $\R$ vanishing at~$0$. The vector field $\partial/\partial x^1$ is not in the image of $\pi_*$, which is then not surjective.

Notice that Theorem~\ref{thm:closed} doesn't apply in this case, since $S$ is not a manifold (it is not locally Euclidean at the origin). So while it is true that $C^\infty(S)^\circ$ is a submanifold algebra of $C^\infty\big(\R^2\big)$ (by Theorem~\ref{thm:closed}), $C^\infty(\R)$ is not a submanifold algebra of $C^\infty(S)^\circ$.
\end{ex}

Consider now an exact sequence \eqref{eq:exactcommutative}, with $B=C^\infty(N)$ the algebra of smooth function on a~second smooth manifold~$N$ (question~\ref{quest:4}). The following theorem holds:

\begin{thm}\label{thm:main} Let $M$ and $N$ be smooth manifolds and $\pi\colon C^\infty(M)\to C^\infty(N)$ a surjective homomorphism. Then $\pi=F^*$ is the pullback of a proper embedding $F\colon N\to M$ and $S:=F(N)$ is a closed embedded submanifold of $M$.
\end{thm}

\begin{proof}By \cite[Corollary~35.10]{KMS}, every homomorphism $\pi\colon C^\infty(M)\to C^\infty(N)$ is the pullback of a smooth map $F\colon N\to M$. Assuming that $\pi=F^*$ is surjective, we will prove first that $F$ is injective, then that it is an immersion, and finally that is a proper embedding.

{\bf $\boldsymbol{F}$ is injective:} by contraddiction, suppose that $F$ is not injective: i.e., $F(p)=F(q)$ for some $p\neq q$ in $N$. Then, every function $f\in C^\infty(N)$ in the image of $F^*$ has $f(p)=f(q)$. Since (for every $p\neq q$) there exists a smooth function with $f(p)\neq f(q)$ (closed disjoint subsets of a~smooth manifold can be separated by a~smooth function), $F^*$ is not surjective.

{\bf $\boldsymbol{F}$ is an immersion:} under the usual identification $T_p\R\simeq\R$ and $T_pf={\rm d}f|_p$ for a scalar function, for all $f\in C^\infty(N)$ and all $p\in N$ one has
\begin{gather}\label{eq:covectors}
{\rm d}f|_p={\rm d}g|_{F(p)}\circ T_pF,
\end{gather}
where $g\in C^\infty(M)$ is any function satisfying $F^*(g)=f$.
It follows from~\eqref{eq:covectors} that the covector ${\rm d}f|_p\in T_p^*N$ vanishes on $\ker T_pF$ for all $f$, which implies that $\ker T_pF=\{0\}$ and $F$ is an immersion at $p$ (covectors ${\rm d}f|_p$ span $T_p^*N$, as one can prove by taking the differential of smooth prolongations of local coordinates).

{\bf $\boldsymbol{F}$ is a proper embedding:}
we proved that $S:=F(N)$ is an immersed submanifold of~$M$. It is well known that the restriction map $C^\infty(M)\to C^\infty(S)$, $f\mapsto f|_S$, to an immersed submanifold is surjective if and only if~$S$ is properly embedded (see, e.g., \cite[Example~5-18(b)]{Lee}), and in particular closed in~$M$.
\end{proof}

An immediate corollary of previous theorem is that, for smooth functions on manifolds, the condition of surjectivity of $\pi_*$ is redundant: if~$C^\infty(N)$ is a quotient algebra of~$C^\infty(M)$, then it is also a submanifold algebra.

Example~\ref{ex:no1} allows us to illustrate a fundamental difference between embeddings and coembeddings.
Embeddings of smooth manifolds satisfy the following property:

\begin{prop}\label{prop:33}
Consider a sequence $N\xrightarrow{\alpha}S\xrightarrow{\beta}M$ of smooth maps between smooth manifolds. If $\beta\circ\alpha$ is a closed embedding, then so is~$\alpha$.
\end{prop}

The proof is straightforward if we rephrase it in terms of algebra morphisms. Recall that every morphism $C^\infty(M)\to C^\infty(S)$ is the pullback of a smooth map $\alpha\colon S\to M$ \cite[Corollary~35.10]{KMS}, and it is surjective if and only if $\alpha$ is a closed embedding (Theorem~\ref{thm:main}). Proposition~\ref{prop:33} is then equivalent to the next Proposition~\ref{prop:34}, whose proof is straightforward:

\begin{prop}\label{prop:34} Consider a sequence $C^\infty(M)\xrightarrow{f} C^\infty(S)\xrightarrow{g}C^\infty(N)$ of algebra morphisms. If $g\circ f$ is surjective, then so is~$g$.
\end{prop}

\begin{proof}The image of $g\circ f$ is a subset of the image of~$g$.
\end{proof}

The noncommutative analogue of Proposition~\ref{prop:34} does not hold. Namely, there are sequences $A\xrightarrow{f}B\xrightarrow{g}C$ of surjective algebra maps such that $f$ and $g\circ f$ are coembeddings but $g$ is not.
Consider for example the sequence $C^\infty\big(\R^2\big)\xrightarrow{f}C^\infty(S)^\circ\xrightarrow{g}C^\infty(\R)$, where $S$ is the cross, $f$ is the pullback of the inclusion $S\hookrightarrow\R^2$ and $g$ the map in Example~\ref{ex:no1}. By Theorem~\ref{thm:closed}, the map~$f$ is a~coembedding, and~$g\circ f$ is the pullback of the inclusion $\R\to\R^2$ as horizontal axis, so it is a~coembedding as well. However~$g$ is not a coembedding, cf.~Example~\ref{ex:no1}.

\subsection{Polynomial algebras}

For commutative algebras an equivalent formulation of the condition of submanifold algebra is via K{\"a}hler differentials.
Let $A$ be a commutative algebra over a field $K$, $M$ an $A$-module and $(\Omega_{A/K},{\rm d})$ the module of K{\"a}hler differentials. Then, the map
\begin{align*}
\operatorname{Hom}_A(\Omega_{A/K},M)
&\xrightarrow{ \smash{\raisebox{-2.5pt}{\ensuremath{\sim}}} }
\operatorname{Der}_K(A,M), \\
f & \mapsto f\circ {\rm d}
\end{align*}
sending $A$-linear maps $\Omega_{A/K}\to M$ to $M$-valued derivations of $A$ ($K$-linear maps $A\to M$ satisfying \eqref{eq:Leibniz}) is a bijection \cite[Chapter~16]{Eis}.
The universal derivation will be denoted always by ${\rm d}$, whatever is the algebra considered.

If $\pi\colon A\to B$ is a homomorphism of commutative algebras,
since ${\rm d}\circ\pi\colon A\to\Omega_{B/K}$ is a derivation,
by the universal property of K{\"a}hler differentials there exists $\pi_{**}\in\operatorname{Hom}_A(\Omega_{A/K},\Omega_{B/K})$
(where we think of $B$-modules as $A$-modules via $\pi$)
 that makes the following diagram commute:
\begin{center}
\begin{tikzpicture}

\node (OA) at (3,1.5) {$\Omega_{A/K}$};
\node (OB) at (3,0) {$\Omega_{B/K}$};
\node (A) at (0,1.5) {$A$};
\node (B) at (0,0) {$B$};

\draw (A) edge[->] node[left] {\footnotesize $\pi$} (B)
 (A) edge[->] node[above] {\footnotesize ${\rm d}$} (OA)
			(B) edge[->] node[above] {\footnotesize ${\rm d}$} (OB)
 (OA) edge[dashed,->] node[right]{\footnotesize $\pi_{**}$} (OB);

\end{tikzpicture}
\end{center}
Explicitly $\pi_{**}(a{\rm d}b)=\pi(a) {\rm d}\pi(b)$ $\forall\, a,b\in A$. Note that if $\pi$ is surjective, $\pi_{**}$ is surjective as well. We can now rephrase the definition of submanifold algebra in terms of K{\"a}hler differential.

\begin{prop}\label{prop:9}
Let $\pi\colon A\to B$ be a surjective homomorphism of commutative algebras. Then~$B$ is a submanifold algebra of $A$ if and only if for every $f\in\operatorname{Hom}_B(\Omega_{B/K},B)$ there exists $\widetilde{f}\in\operatorname{Hom}_A(\Omega_{A/K},A)$ that makes the following diagram commute:
\begin{center}
\begin{tikzpicture}

\node (OA) at (0,1.5) {$\Omega_{A/K}$};
\node (OB) at (0,0) {$\Omega_{B/K}$};
\node (A) at (3,1.5) {$A$};
\node (B) at (3,0) {$B$};

\draw (A) edge[->] node[right] {\footnotesize $\pi$} (B)
 (OA) edge[dashed,->] node[above] {\footnotesize $\widetilde{f}$} (A)
			(OB) edge[->] node[above] {\footnotesize $f$} (B)
 (OA) edge[->] node[left]{\footnotesize $\pi_{**}$} (OB);

\end{tikzpicture}
\end{center}
\end{prop}

\begin{proof}
The situation is illustrated in the following diagram:
\begin{center}
\begin{tikzpicture}
\node (A) at (0,1.5) {$A$};
\node (B) at (0,0) {$B$};
\node (OA) at (3,1.5) {$\Omega_{A/K}$};
\node (OB) at (3,0) {$\Omega_{B/K}$};
\node (AA) at (6,1.5) {$A$};
\node (BB) at (6,0) {$B$};

\draw (A) edge[->] node[left] {\footnotesize $\pi$} (B)
 (A) edge[->] node[above] {\footnotesize ${\rm d}$} (OA)
			(B) edge[->] node[above] {\footnotesize ${\rm d}$} (OB)
 (OA) edge[->] node[right]{\footnotesize $\pi_{**}$} (OB)
 (OA) edge[->] node[above] {\footnotesize $\widetilde{f}$} (AA)
			(OB) edge[->] node[above] {\footnotesize $f$} (BB)
			(A) edge[bend left,->] node[above]{\footnotesize $\widetilde{D}$} (AA)
			(B) edge[bend right,->] node[below]{\footnotesize $D$} (BB)
			(AA) edge[->] node[right] {\footnotesize $\pi$} (BB);

\end{tikzpicture}
\end{center}
Suppose every $D\in\operatorname{Der}(B)$ admits a lift $\widetilde{D}$ making the outer diagram commutative (automatically $\widetilde{D}(J)\subset J$ and $\widetilde{D}\in\operatorname{Der}_\pi(A)$). Then for every $f\in\operatorname{Hom}_B(\Omega_{B/K},B)$ we can lift the derivation $D=f\circ {\rm d}$ to a derivation $\widetilde{D}$ of $A$. By universality $\widetilde{D}=\widetilde{f}\circ {\rm d}$ for some $\widetilde{f}\in\operatorname{Hom}_A(\Omega_{A/K},A)$, and $\widetilde{f}$ is the lift of~$f$ we are looking for.

Vice versa, if every $f$ admits a lift $\widetilde{f}$, given any $D\in\operatorname{Der}(B)$ and given $f$ such that $f\circ {\rm d}=D$, we can lift $f$ to $\widetilde{f}$ and form the derivation $\widetilde{D}=\widetilde{f}\circ {\rm d}$. Automatically $\pi_*(\widetilde{D})=D$, so that $\pi_*$ is surjective.
\end{proof}

\begin{prop}If $A$ is commutative and $\Omega_{A/K}$ is a free $A$-module, then every quotient algebra of~$A$ is a submanifold algebra.
\end{prop}

\begin{proof}It follows from Proposition~\ref{prop:9}. Let $S$ be a free generating set for $\Omega_{A/K}$. Then, every K{\"a}hler differential can be written in a unique way as a finite sum $\sum\limits_{\xi\in S}a_\xi\cdot\xi$ with $a_\xi\in A$. Let $f\in\operatorname{Hom}_B(\Omega_{B/K},B)$. For every $\xi\in S$ there exists (by surjectivity of $\pi$) an element $\widetilde{\xi}\in A$ such that $\pi\big(\widetilde{\xi}\big)=f\circ\pi_{**}(\xi)$. Let $\widetilde{f}$ be defined by
\begin{gather*}
\widetilde{f}\bigg(\sum_{\xi\in S}a_\xi\cdot\xi\bigg):=\sum_{\xi\in S}a_\xi\cdot \widetilde{\xi}.
\end{gather*}
By construction $\widetilde{f}\in\operatorname{Hom}_A(\Omega_{A/K},A)$ and $\pi\circ\widetilde{f}=f\circ\pi_{**}$ as requested.
\end{proof}

\begin{ex}[affine sets]
Let $A:=K[x_1,\ldots,x_n]$, $S\subset K^n$ an affine algebraic set, $J$ the radical ideal of $S$, $B:=A/J$ its coordinate ring and $\pi\colon A\to B$ the quotient map. Since~$\Omega_{A/K}$ is a free $A$-module \cite[Proposition~16.1]{Eis}, $B$ is a submanifold algebra of~$A$.
\end{ex}

Similarly to the case of smooth functions, the coordinate algebra $B$ of $S$ in previous example is a submanifold algebra even when~$S$ is not a smooth manifold. In the case of affine varieties there is however a simple criteria to know if~$S$ is smooth by looking at derivations: if $K$ is algebraically closed and $S$ is irreducible, then~$S$ is a smooth submanifold of $K^n$ if and only if the Lie algebra $\operatorname{Der}(B)$ is simple (see, e.g.,~\cite{BF}).

We close this section with an example of quotient of a polynomial algebra that is not a~submanifold algebra (the polynomial version of Example~\ref{ex:no1}).

\begin{ex}[the ordinary double point]\label{ex:cross2}
Let $A:=K[x,y]/\ideal{xy}$, $B:=K[x]$ and $\pi$ the surjective homomorphism defined by $\pi(y)=0$. A simple computation shows that: $\operatorname{Der}(A)$ is generated by the two derivations $x\frac{\partial}{\partial x}$ and $y\frac{\partial}{\partial y}$; $\operatorname{Der}_\pi(A)=\operatorname{Der}(A)$; the image of $\pi_*$ is freely generated (as a $B$-module) by $x\frac{{\rm d}}{{\rm d}x}$; $\operatorname{Der}(B)$ is freely generated by $\frac{{\rm d}}{{\rm d}x}$. The derivation $\frac{{\rm d}}{{\rm d}x}$ is not in the image of $\pi_*$, that is then not surjective.
\end{ex}

\section{Noncommutative submanifolds: examples}\label{sec:ncexamples}

In parallel with previous section, let's start with (separable complex) $C^*$-algebras. In this case, it is well known that any derivation of a quotient algebra can be lifted~\cite{Ped}. Thus:

\begin{ex}[$C^*$-algebras]
If $\pi\colon A \to B$ is a surjective morphism between separable complex $C^*$-algebras, then $\pi_*$ is surjective.
\end{ex}

In Section~\ref{sec:CinftyM} we studied examples with no non-zero inner derivations. On the other side of the spectrum, any quotient algebra with only inner derivations is a submanifold algebra (Remark~\ref{rem:inner}). This includes the universal enveloping algebra of a semisimple finite-dimensional Lie algebra (over a field $K$ of characteristic~$0$), any central simple algebra (as a corollary of Skolem--Noether theorem, cf.~\cite[p.~80, Example~4c]{DK}), Weyl algebras \cite[Lemma~4.6.8]{Dix}, all von Neumann algebras (Sakai--Kadison theorem~\cite{Kad,Sak}) and in particular finite-dimensional complex $C^*$-algebras (since they coincide with their weak closure). It is not difficult to extend the latter result to the real case. Let us record the result for future use.

\begin{lem}\label{lemma:finitedim}Let $B$ be a finite-dimensional $($real or complex$)$ $C^*$-algebra. Then all derivations of $B$ are inner.
\end{lem}

\begin{proof}We need to prove the statement in the real case. Assume then that $B$ is real, so that $B_{\C}=B+{\rm i}B$ is a complex finite-dimensional $C^*$-algebra. Every $\R$-linear endomorphism $\phi$ of the vector space underlying $B$ can be extended to a $\C$-linear endomorphism $\phi_{\C}$ of $B_{\C}$ by
\begin{gather*}
\phi_{\C}(a+{\rm i}b)=\phi(a)+{\rm i}\phi(b),\qquad\forall\, a,b\in B.
\end{gather*}
It is straightforward to check that if $\phi\in\operatorname{Der}(B)$ then $\phi_{\C}\in\operatorname{Der}(B_{\C})$. Since $B_{\C}$ has only inner derivations, $\phi_{\C}=[x+{\rm i}y,\,\cdot\,]$ for some fixed $x,y\in B$. But then for all $a\in B$:
\begin{gather*}
\phi(a)=\phi_{\C}(a)=[x,a]+{\rm i}[y,a].
\end{gather*}
Since $\phi$ has image in $B$, $y$ must be central, and $\phi(a)=[x,a]$ for all $a\in B$.
\end{proof}

As a corollary:

\begin{ex}\looseness=-1 Let $\pi\colon A\to B$ be a surjective homomorphism with~$B$ a finite-dimensional real or complex $C^*$-algebra (and $A$ any real or complex algebra). Then~$B$ is a submanifold algebra of~$A$.
\end{ex}

In this list of sporadic examples, the next is a noncommutative example of quotient algebra that is not a submanifold algebra.

\begin{ex}\label{ex:no2}
Let $A:=U(\mathfrak{sb}(2,\R))$ be the universal enveloping algebra of the Lie algebra spanned by two elements $H$ and $E$ with relation $[H,E]=E$. Let $B:=\R[x]$ be the algebra of polynomials in an indeterminate~$x$. By Peter--Weyl theorem $A$ has basis $\big\{H^jE^k\big\}_{j,k\geq 0}$. The map $\pi\colon A\to B$ defined by
\begin{gather*}
\pi\big(H^jE^k\big)=x^j\delta_{k,0}
\end{gather*}
is a surjective homomorphism. The kernel $J$ has basis $\big\{H^jE^k\big\}_{j\geq 0,k\geq 1}$ and is generated by $E$. From
\begin{gather*}
\big[H,\tfrac{1}{k}H^jE^k\big]=H^jE^k \qquad\forall\, j\geq 0,k\geq 1
\end{gather*}
we deduce that $J$ is contained in the commutator ideal, and since $B\simeq A/J$ is commutative $J=[A,A]$ must be exactly the commutator ideal.

Let us prove that the map $\pi_*$ is not surjective. Let $\widetilde{D}\in\operatorname{Der}(B)$ be the derivation $\widetilde{D}=x\frac{{\rm d}}{{\rm d}x}$ and suppose $\widetilde{D}=\pi_*(D)$ for some $D$. Set
\begin{gather*}
D(H)=\sum_{j,k\geq 0}a_{j,k}H^jE^k \qquad\text{and}\qquad
D(E)=\sum_{j,k\geq 0}b_{j,k}H^jE^k
\end{gather*}
with $a_{j,k},b_{j,k}\in\R$.
From $\widetilde{D}=\pi_*(D)$ we deduce that $a_{1,0}=1$ and $a_{j,0}=0$ for all $j\neq 1$. From $D(E)\in J$ we deduce that $b_{j,0}=0$ $\forall\, j\geq 0$. From the Leibniz rule we get
\begin{align*}
0 &=D([H,E]-E)=[D(H),E]+[H,D(E)]-D(E) \\
&=\bigg[H+\sum_{j\geq 0,\, k\geq 1}a_{j,k}H^jE^k,E\bigg]+\bigg[H,
\sum_{j\geq 0,\, k\geq 1}b_{j,k}H^jE^k\bigg]-\sum_{j\geq 0,k\geq 1}b_{j,k}H^jE^k \\
&=E+
\sum_{j\geq 0,\, k\geq 1}a_{j,k}\big[H^j,E\big]E^k
+\sum_{j\geq 0,\, k\geq 2}b_{j,k}(k-1)H^jE^k.
\end{align*}
Since $\big[H^j,E\big]\in J$, we get $0=E+({\cdots})E^2$, which is not zero whatever is the element that multiplies $E^2$. We arrived at a contraddiction, proving that $\widetilde{D}\notin\operatorname{Im}(\pi_*)$.
\end{ex}

Example~\ref{ex:no2} has a ``folkloristic'' interpretation: (the complexification of) the algebra $A$ can be interpreted as the coordinate algebra of $\kappa$-Minkowski space~\cite{kappa} in $1+1$ dimensions, that is generated by a time operator $t$ and position operator $q$ subject to the relation $[t,q]={\rm i}\kappa^{-1}q$ (with $\kappa\neq 0$ a constant). The algebra $B$ can be interpreted as the coordinate algebra on the time ``axis''. Such an axis is not a submanifold of $\kappa$-Minkowski in the sense of Definition~\ref{def:sMA}.

\subsection{Free algebras}
Any homomorphism $\pi\colon A\to B$ maps the center $Z(A)$ of $A$ into the center $Z(B)$ of $B$. If $\pi$ is surjective,
$\pi_*\colon \operatorname{Der}_\pi(A)\to\operatorname{Der}(B)$ is a module map that covers the map $\pi|_{Z(A)}\colon Z(A)\to Z(B)$, meaning that
\begin{gather*}
\pi_*(zD)=\pi(z)\pi_*(D)\qquad\forall\, D\in\operatorname{Der}_\pi(A)\qquad \text{and} \qquad z\in Z(A).
\end{gather*}
In Example \ref{ex:no2} the map $\pi|_{Z(A)}\colon Z(A)\to Z(B)$ is not surjective (since $Z(A)=K$ is trivial while $Z(B)=B$ is not). One may think that this condition has something to do with non-surjectivity of~$\pi_*$, but next example shows that this is not the case.

\begin{ex}[free algebras]
Let \mbox{$A=K\!\inner{x,y}$} be the free algebra with two generators~$x$ and~$y$, $B=K[x,y]$, $J=[A,A]$ the commutator ideal. Note that every derivation $D$ sends the commutator ideal into itself, since
\begin{gather*}
D([a,b])=[D(a),b]+[a,D(b)] \qquad\forall\, a,b\in A.
\end{gather*}
Thus $\operatorname{Der}_\pi(A)=\operatorname{Der}(A)$. Since $Z(A)=K$ is the set of constant polynomials and $Z(B)=B$, clearly $\pi|_{Z(A)}\colon Z(A)\to Z(B)$ is not surjective.

If we pick up two elements $a,b\in A$, there exists a unique derivation $D$ such that $D(x)=a$ and $D(y)=b$ (by universality, extended linearly and via the Leibniz rule). Thus $\operatorname{Der}(A)\simeq A^2$ as a~vector space. For such a $D$, $\pi_*(D)=\pi(a)\frac{\partial}{\partial x}+\pi(b)\frac{\partial}{\partial y}$, hence~$\pi_*$ is surjective.
\end{ex}

In fact, any quotient algebra of a free algebra is a submanifold algebra.
The tensor algebra example in~\cite{Mas} is a special case of Proposition~\ref{prop:free}.

\begin{prop}\label{prop:free} Let $S$ be a set, $A:=K\!\inner{S}$ the free algebra generated by $S$ and $\pi\colon A\to B$ a~surjective homomorphism to a second associative algebra~$B$. Then~$B$ is a submanifold algebra of~$A$.
\end{prop}

\begin{proof}
By the universal property, for every map $f\colon S\to C$ from $S$ to an associative algebra $C$ there exists a unique homomorphism $\widetilde{f}\colon A\to C$ such that $\widetilde{f}|_S=f$.
Let $D\in\operatorname{Der}(B)$. For all $x\in S$ there exists
$\delta_x\in A$ such that $\pi(\delta_x)=D(\pi(x))$ (by surjectivity of~$\pi$). Let
$C\subset M_2(A)$ be the subalgebra of matrices of the form
\begin{gather*}
\mat{a & b \\ 0 & a}, \qquad a,b\in A.
\end{gather*}
Let $f\colon S\to C$ be the map given by
\begin{gather*}
f(x):=\mat{ x & \delta_x \\ 0 & x }
\end{gather*}
and $\widetilde{f}\colon A\to C$ the corresponding algebra morphism. For $i=1,2$ denote by $\operatorname{pr}_i\colon C\to A$ the projection on the matrix element in position $(i,i)$. Since $\operatorname{pr}_i\circ\widetilde{f}\colon A\to A$ is a homomorphism given on $S$ by $\operatorname{pr}_i\circ f=\id_S$, by unicity of the lift it must be $\operatorname{pr}_i\circ\widetilde{f}=\id_A$. Therefore
\begin{gather*}
\widetilde{f}(a)=\mat{ a & \widetilde{D}(a) \\ 0 & a }
\end{gather*}
for some map $\widetilde{D}\colon A\to A$. Since $\widetilde{f}$ is a homomorphism, from Lemma~\ref{lemma:17} we deduce that~$\widetilde{D}$ is a~derivation.

Extend $\pi$ to $M_2(A)$ in the obvious way. By construction
\begin{gather*}
\pi(\widetilde{f}(x))=\mat{ \pi(x) & \pi(\delta_x) \\ 0 & \pi(x) }=
\mat{ \pi(x) & D(\pi(x)) \\ 0 & \pi(x) }=f(\pi(x))
\end{gather*}
for all $x\in S$. Since $\pi\circ\widetilde{f}$ and $f\circ\pi$ are homomorphism that coincide on generators, they must be equal, which means $\pi\circ\widetilde{D}=D\circ\pi$. The latter automatically implies that $\widetilde{D}\in\operatorname{Der}_\pi(A)$.
\end{proof}

\subsection{Almost commutative spaces}
One of the main applications of vector bundles in physics is to Yang--Mills theories: here $L^2$ sections of a complex smooth Hermitian vector bundle $\pi\colon E\to M$ describe the physical state of a particle (or several particles) ``living'' in the manifold $M$; vectors in a fiber describe the ``internal'' degrees of freedom of such a particle. In the celebrated Dirac equation, for example, $M$ is a $4$-dimensional Riemannian spin manifold, $\pi\colon E\to M$ the spinor bundle and one looks for solutions of the equation -- among smooth sections of such a bundle~-- describing the state of a couple particle-antiparticle with spin $1/2$.

Inspired by Kaluza--Klein theories, where one derives a $4$-dimensional Yang--Mills theory coupled with gravity from a purely gravitational theory on some auxiliary $5$-dimensional mani\-fold, A.~Connes suggested to replace the unobserved extra dimension of Kaluza--Klein by a~$0$-dimensional noncommutative space. The starting point is the tensor product $C^\infty(M)\otimes F$ of smooth functions on a manifold and a finite-dimensional real algebra~$F$, describing some kind of virtual $0$-dimensional noncommutative space.
Starting from a purely geometric theory on such a product, one is able to derive the complicated Lagrangian of the standard model of particle physics coupled with gravity\footnote{This is of course an oversimplification: the full story is beyond the scope of this paper. The interested reader can consult the books \cite{CM,vS}.} (see, e.g.,~\cite{CM} or~\cite{vS} and references therein).

Following the point of view of Connes, an ``almost commutative space'' is something described by the tensor product of the algebra of smooth functions on a manifold and some finite-dimensional algebra or, more generally, by an \emph{algebra bundle} over a manifold.

Let $K=\R$ or $\C$ and $F$ be a finite-dimensional $K$-algebra. A smooth \emph{algebra bundle} over a~(real) smooth manifold $M$, with typical fiber $F$, is a smooth vector bundle $\pi\colon E\to M$ whose fibers are $K$-algebras and whose local trivializations give maps \mbox{$E_p\to\{p\}\times F$} ($\forall\, p\in M$) that are not only isomorphisms of $K$-vector spaces, but of $K$-algebras as well \cite[p.~377]{Gre}.

Given any smooth vector bundle $\pi\colon E\to M$,
it is well known that $M$ is a submanifold of~$E$ (via the zero section); moreover if $S\subset M$ is a submanifold, then $\pi^{-1}(S)\subset E$ is a~submanifold (inverse image of a~submanifold by means of a submersion) and in particular all fibers \mbox{$E_p=\pi^{-1}(p)$} are submanifolds of~$E$.
Having at our disposal an algebraic notion of submanifold, we wonder if analogous properties hold for algebra bundles.

If $\pi\colon E\to M$ is an algebra bundle, the module $\Gamma^\infty(\pi)$ of global smooth sections is a $K$-algebra with pointwise product. For $\xi,\eta\in\Gamma^\infty(\pi)$ we define
\begin{gather*}
(\xi\cdot\eta)(p):=\xi(p)\eta(p) \qquad \forall\, p\in M,
\end{gather*}
where the one on the right is the product in the fiber $E_p$. By construction for any $p\in M$, the evaluation at $p$ gives a homomorphism $\operatorname{ev}_p\colon \Gamma^\infty(\pi)\to E_p$.

If $E=M\times F$ and $\pi$ is the projection on the first factor, then $\Gamma^\infty(\pi)\simeq C^\infty(M)\otimes F$ similarly to the standard model example.

\begin{prop}\label{prop:20}
Let $\pi\colon E\to M$ be an algebra bundle with typical fiber a finite-dimensional real or complex $C^*$-algebra.\footnote{For example $F=\C\oplus\mathbb{H}\oplus M_3(\C)$ in the case of the standard model, where $\mathbb{H}$ is the (real) algebra of quaternions.} Then, for every $p\in M$, the map
\begin{gather*}
\operatorname{ev}_p\colon \ \Gamma^\infty(\pi)\to E_p
\end{gather*}
is a coembedding.
\end{prop}

\begin{proof}It is true for every vector bundle that any vector in a fiber can be extended to a~global smooth section. The map $\operatorname{ev}_p$ is then surjective. Since every derivation of~$E_p$ is inner (Lem\-ma~\ref{lemma:finitedim}), the induced map on derivations is surjective as well (Re\-mark~\ref{rem:inner}).
\end{proof}

In the example of the standard model, previous proposition can be interpreted by saying that the $0$-dimensional noncommutative space encoding the internal degrees of freedom of particles is a ``noncommutative submanifold'' of the product space.
One may wonder if $M$ is a ``noncommutative submanifold'' as well: it is difficult to answer such a question in general, since a homomorphism $\Gamma^\infty(\pi)\to C^\infty(M)$ may not even exist. We will investigate this question for trivial algebra bundles, i.e., tensor products of algebras, cf.~Example~\ref{ex:products}.

Another example covered by Proposition~\ref{prop:20} is the rational noncommutative torus. Let $\theta=p/q\in\mathbb{Q}$ be a rational number, with $p$ and $q$ coprime. The algebra of ``complex-valued smooth functions'' on the noncommutative torus $\mathbb{T}_\theta$ is isomorphic to $\Gamma^\infty(\pi)$ with $\pi\colon E\to \mathbb{T}^2$ a~suitable algebra bundle over the (ordinary) $2$-torus~\cite{GVF}. The typical fiber is the algebra~$M_q(\C)$ of all $q\times q$ complex matrices. The spectral triple of the rational noncommutative torus was recently studied from the point of view of algebra bundles in \cite{CD}.

\begin{ex}[tensor products]\label{ex:products}
Let $K$ be any field and $A:=A_1\otimes A_2$ a tensor product of two associative $K$-algebras. Suppose $\varepsilon\colon A_1\to K$ is a non-zero augmentation. Then $\pi:=\varepsilon\otimes\id\colon A\to A_2$ is a surjective homomorphism. For every $D\in\operatorname{Der}(A_2)$, the formula
\begin{gather*}
\widetilde{D}(a_1\otimes a_2):=a_1\otimes D(a_2),\qquad\forall\, a_1\in A_1,a_2\in A_2,
\end{gather*}
defines a derivation $\widetilde{D}\in\operatorname{Der}_\pi(A)$ satisfying by construction $\pi_*(\widetilde{D})=D$. Thus, $A_2$ is a submanifold algebra of $A$.
\end{ex}

\section{Formal deformations}\label{sec:starproducts}
\noindent
A rich source of examples of ``noncommutative spaces'' is from deformation quantizations of Poisson manifolds. Only in this section, $C^\infty(M)$ will denote \emph{complex}-valued smooth functions on a (real) smooth manifold $M$.

\begin{df}\label{def:star}
A \emph{star product} on a Poisson manifold $(M,\{,\})$ is a $\C\h$-bilinear associative binary operation $\star$ on $C^\infty(M)\h$ of the form:
\begin{gather}\label{eq:star}
f\star g=\sum_{k=0}^\infty \hslash^k C_k(f,g),\qquad\forall\, f,g\in C^\infty(M),
\end{gather}
where each \mbox{$C_k\colon C^\infty(M)\times C^\infty(M)\to C^\infty(M)$} is a bi-differential operator
and for all $f,g\in C^\infty(M)$:
\begin{alignat*}{3}
& C_0(f,g) =fg \qquad && \text{is the pointwise multiplication},& \\
& C_1(f,g)-C_1(g,f) =2\mathrm{i}\left\{f,g\right\} \qquad && \text{is the Poisson bracket},& \\
&C_k(1,f)=0  \ \forall\, k\geq 1 \qquad && \text{($1$ is a neutral element for $\star$)}.&
\end{alignat*}
\end{df}

\noindent
From now on we will always assume that $C_1$ is antisymmetric,\footnote{This can be done without loss of generality: any star product is equivalent to one with $C_1$ antisymmetric \cite[Proposition~2.23]{GR}.} so that
\begin{gather*}
C_1(f,g)=\mathrm{i}\left\{f,g\right\}.
\end{gather*}
If we stop the sum \eqref{eq:star} at order $r\geq 1$ and work over the ring $\C[\hslash]/\ideal{\hslash^{r+1}}$ we get the notion of \emph{order $r$ deformation} of a Poisson manifold.

In the framework of deformation quantization, we can consider the problem of star products that are \emph{tangential} to submanifolds, and investigate under what conditions every derivation of the star product on the submanifold admits a prolongation (cf.~Section~\ref{sec:Poissonsub}). For 1st order deformations this becomes a problem of prolongation of Poisson vector fields, i.e., vector fields $X\in\mathfrak{X}(M)$ satisfying
\begin{gather*}
X\big(\{f,g\}\big)=\big\{X(f),g\big\}+\big\{f,X(g)\big\},\qquad\forall\, f,g\in C^\infty(M).
\end{gather*}
In Section~\ref{sec:coiso} we will consider short exact sequences of formal deformations coming from coisotropic reduction of a Poisson manifold.

\subsection[$\hslash$-linear derivations]{$\boldsymbol{\hslash}$-linear derivations}\label{sec:localh}
Let $A:=\big(C^\infty(M)\h,\star\big)$ be a deformation quantization of a Poisson manifold~$M$. When dealing with star products, we will only consider derivations of star products that are \mbox{$\hslash$-linear}.
An element of $\operatorname{Der}(A)$ will be then a formal power series
\begin{gather*}
D=\sum_{k=0}^\infty\hslash^kD_k
\end{gather*}
of differential operators on $M$ satisfying the Leibniz rule:
\begin{gather*}
D(f\star g)=(Df)\star g+f\star (Dg),\qquad\forall\, f,g\in C^\infty(M).
\end{gather*}
At order $0$ in $\hslash$ this means that $D_0$ is a derivation of $C^\infty(M)$, i.e., a vector field. At order $1$ we get
\begin{gather*}
D_1(fg)-(D_1f)g-f(D_1g)=C_1(D_0f, g)+C_1(f,D_0g)-D_0C_1(f,g).
\end{gather*}
Since in previous equality the left hand side is symmetric and the right hand side antisymmetric, we deduce that they must both vanish. Thus, $D_1$ must be a vector field as well, and $D_0$ must be a~Poisson vector field. If we are interested in first order deformations, this completely characterizes derivations.

\begin{lem}An $\varepsilon$-linear derivation of $\big(C^\infty(M)[\varepsilon]/\ideal{\varepsilon^2},\star\big)$ is a sum
\begin{gather*}
D_0+\varepsilon D_1
\end{gather*}
of a Poisson vector field $D_0$ and an arbitrary vector field $D_1$ on $M$.
\end{lem}

On a symplectic manifold, the correspondence between derivations of a star product and Poisson (in this case \emph{symplectic}) vector fields holds at any order in~$\hslash$. Suppose $M$ is a symplectic manifold and denote by $\mathfrak{X}^{\mathrm{sym}}(M)\h$ the space of \emph{formal symplectic vector fields} on~$M$. Elements of $\mathfrak{X}^{\mathrm{sym}}(M)\h$ are formal power series
\begin{gather*}
X=\sum_{k=0}^\infty\hslash^kX_k
\end{gather*}
with $X_k$ a symplectic vector field on $M$ for all $k\geq 0$. Every symplectic vector field on a~contractible open set is Hamiltonian. We can then cover $M$ by contractible open subsets, and on each $U$ of this cover find functions $f_k^U\in C^\infty(U)$ such that
\begin{gather*}
X_k(g)|_U=\big\{f_k^U,g\big\}
\end{gather*}
for all $g\in C^\infty(M)$ and all $k\geq 0$. These functions are determined by $X$ up to an additive constant ($U$ is connected). We can then define a new function $Dg\in C^\infty(M)$ given, on each set~$U$ of this cover, by
\begin{gather*}
Dg|_U=\frac{1}{\hslash}\sum_{k\geq 0}\hslash^k\big(f_k^U\star g-g\star f_k^U\big).
\end{gather*}
For all $g$ we get a well-defined formal power series of global smooth functions~$Dg$, and a well define derivation $D$ of the star product. Such a derivation depends only on~$X$. A simple argument by induction shows that every derivation of the star product is in fact of this form:

\begin{thm}[{\cite[Proposition~3.5]{GR}}]\label{thm:24}
Every $\hslash$-linear derivation of a star product on a symplectic manifold $M$ corresponds to a formal symplectic vector field via the construction above.
\end{thm}

If $H^1(M,\R)=0$ every symplectic vector field is Hamiltonian, and every derivation of the star product is essentially inner,\footnote{Inner except for the factor $\frac{1}{\hslash}$ in front.} given by $\frac{1}{\hslash}\left[f\stackrel{\star}{,} \right]$ for some $f\in C^\infty(M)\h$.

\subsection{Deformations of Poisson submanifolds}\label{sec:Poissonsub}
Suppose
\begin{gather}\label{eq:formalpi}
\pi\colon \ \big(C^\infty(M)\h,\star\big)\to \big(C^\infty(N)\h,\star\big)
\end{gather}
is a homomorphisms between deformation quantizations of two Poisson manifolds $M$ and $N$. We will assume that $\pi$ is $\hslash$-linear, i.e., of the form
\begin{gather}\label{eq:formalpiseries}
\pi=\sum_{k=0}^\infty\hslash^k\pi_k
\end{gather}
where each $\pi_k$ maps $C^\infty(M)$ to $C^\infty(N)$ and is extended to formal power series by $\C\h$-linearity.

If we look at the condition $\pi(f\star g)=\pi(f)\star\pi(g)$ at order $0$ we get that $\pi_0$ must be a~homomorphism between the commutative algebras $C^\infty(M)$ and $C^\infty(N)$, hence the pullback of a~smooth map $\varphi_0\colon N\to M$ \cite[Corollary~35.10]{KMS}. At order $1$ we get:
\begin{gather*}
\varphi_0^*\big(\{f,g\}\big)-\left\{\varphi_0^*f,\varphi_0^*g\right\}=
\pi_1(f)\pi_1(g)-\pi_1(fg).
\end{gather*}
Since the symmetric and antisymmetric part must both vanish, we deduce that $\varphi_0$ is a Poisson map and $\pi_1$ a homomorphism, hence the pullback of a smooth map $\varphi_1\colon N\to M$.

\begin{lem}The map \eqref{eq:formalpiseries} is surjective if and only if $\pi_0\colon C^\infty(M)\to C^\infty(N)$ is surjective.
\end{lem}

\begin{proof}
``$\Rightarrow$''
If $\pi$ is surjective, for all $g=\sum\limits_{k=0}^\infty\hslash^kg_k$, with $g_k\in C^\infty(N)$, there exists $f=\sum\limits_{k=0}^\infty\hslash^kf_k$, with $f_k\in C^\infty(M)$, such that
\begin{gather*}
\pi(f)=\pi_0(f_0)+O(\hslash)=g=g_0+O(\hslash).
\end{gather*}
Thus $\pi_0(f_0)=g_0$ and $\pi_0$ is surjective.

``$\Leftarrow$''
Suppose $\pi_0$ is surjective and let $g=\sum\limits_{k=0}^\infty\hslash^kg_k$, with $g_k\in C^\infty(N)$. It follows from surjectivity of $\pi_0$ that the recursive equation
\begin{gather*}
\pi_0(f_k)=g_k-\sum_{j=1}^k\pi_j(f_{k-j})
\end{gather*}
admits a solution $f=\sum\limits_{k=0}^\infty\hslash^kf_k$, with $f_k\in C^\infty(M)$. Such a formal power series satisfies by construction $\pi(f)=g$, hence $\pi$ is surjective.
\end{proof}

If $\pi_0$ is surjective, it follows from Theorem \ref{thm:main} that $\varphi_0(N)$ is a closed embedded Poisson submanifold of $M$.

Conversely, suppose $S$ is a closed embedded submanifold of a Poisson manifold $M$ and $J$ is the vanishing ideal of $S$. If $J$ is a \emph{Poisson ideal}, i.e., $\{J,f\}\subset J$ for all $f\in C^\infty(M)$, the Poisson structure of $M$ induces one on $S$ and $S$ becomes a Poisson submanifold of $M$.\footnote{Note that $J$ being a Poisson ideal means that all Hamiltonian vector fields $X_f:=\{f,\,\cdot\,\}$ belong to $\operatorname{Der}_{\imath^*}(C^\infty(M))$.}

A star product \eqref{eq:star} on $M$ is \emph{tangential} to $S$ if $C_k(J,f)\subset J$ for all $k\geq 1$ and for all $f\in C^\infty(M)$. If such a condition is satisfied, we get by restriction a star product on $S$. In such a situation, the pullback of the inclusion
$\imath\colon S\to M$ extends by $\C\h$-linearity to a surjective homomorphism
\begin{gather}\label{eq:formalder}
\imath^*\colon \ A:=\big(C^\infty(M)\h,\star\big)\to B:=\big(C^\infty(S)\h,\star\big).
\end{gather}
This is an instance of surjective homomorphism~\eqref{eq:formalpi} where $\pi=\pi_0=\imath^*$ has no higher order terms.
Among the examples in this class we find regular coadjoint orbits of compact Lie groups.

\begin{ex}[regular coadjoint orbits] Let $G$ be a compact connected Lie group, $\imath\colon \mathcal{O}\hookrightarrow\mathfrak{g}^*$ a regular coadjoint orbit, and equip $C^\infty(\mathfrak{g}^*)\h$ with BCH star product $\star_{\hspace{1pt}\text{\tiny BCH}}$. Then there exists a star product $\star$ on $\mathcal{O}$ and a series of $\mathfrak{g}$-invariant differential operators $T=\id+\sum\limits_{k=1}^\infty\hslash^kT_k$ on~$\mathfrak{g}^*$ such that
\begin{gather*}
\imath^*\circ T\colon \ \big(C^\infty(\mathfrak{g}^*)\h,\star_{\hspace{1pt}\text{\tiny BCH}}\big)\to \big(C^\infty(\mathcal{O})\h,\star\big)
\end{gather*}
is a homomorphism \cite[Theorem~5.2]{BBGW}. If we define a new star product $\star$ on $\mathfrak{g}^*$ (equivariantly equivalent to $\star_{\hspace{1pt}\text{\tiny BCH}}$) by
\begin{gather*}
f\star g := T\big(T^{-1}f\star_{\hspace{1pt}\text{\tiny BCH}} T^{-1}g\big),
\end{gather*}
then the new star product is tangential to $\mathcal{O}$, and
\begin{gather*}
\pi:=\imath^*\colon \ \big(C^\infty(\mathfrak{g}^*)\h,\star\big)\to \big(C^\infty(\mathcal{O})\h,\star\big)
\end{gather*}
is a surjective homomorphism. The coadjoint orbits of a compact connected Lie group are simply connected \cite[Theorem~2.3.7]{Fil}, thus $H^1(\mathcal{O},\R)=0$, derivations of $(C^\infty(\mathcal{O})\h,\star)$ are essentially inner and $\pi_*$ is surjective.
\end{ex}

One could conjecture that~\eqref{eq:formalder} is always a coembedding if $S$ is symplectic, since at least locally the derivations are essentially inner (Theorem~\ref{thm:24}). Unfortunately this is not the case, as shown by the next lemma.

\begin{lem}[obstructions]\label{lemma:obstr}
Consider a surjective homomorphism like in \eqref{eq:formalder}.
\begin{itemize}\itemsep=0pt
\item[$(a)$] If there is a Poisson vector field on $S$ that cannot be extended to a~Poisson vector field on~$M$, then \eqref{eq:formalder} is not a coembedding.
\end{itemize}
Assume that both $M$ and $S$ are symplectic. Then:
\begin{itemize}\itemsep=0pt
\item[$(b)$] the morphism \eqref{eq:formalder} induces a linear map:
\begin{gather}\label{eq:H1map}
H^1(M,\R)\to H^1(S,\R).
\end{gather}
If \eqref{eq:formalder} is a coembedding, the map~\eqref{eq:H1map} is surjective.
\item[$(c)$] If $H^1(M,\R)=0$, then \eqref{eq:formalder} is a coembedding if and only if $H^1(S,\R)=0$.
\end{itemize}
\end{lem}

\begin{proof}
(a) Let $\pi=\imath^*$ be the pullback of the inclusion as in \eqref{eq:formalder}, and assume that $\pi_*$ is surjective on derivations (of the star products). For any Poisson vector field $\widetilde{D}_0$ on $S$
there is a derivation $\widetilde{D}$ of the star product on $S$ such that $\widetilde{D}=\widetilde{D}_0+O(\hslash)$ \cite[Proposition~3.1]{Sha}. By hypothesis, there exists a derivation $D=D_0+O(\hslash)$ of the star product on $M$ such that
\begin{gather*}
\pi_*(D)=\widetilde{D}.
\end{gather*}
This in particular means that $\widetilde{D}_0$ is the restriction to $S$ of the vector field $D_0$ on $M$. Being the zeroth order part of a derivation (of a star product), $D_0$ is a Poisson vector field on $M$. Thus, if~\eqref{eq:formalder} is a coembedding, every Poisson vector field $\widetilde{D}_0$ on $S$ can be extended to a Poisson vector field $D_0$ on~$M$.

(b) Recall that on a symplectic manifold $M$ the 1st de Rham cohomology is isomorphic to the quotient of symplectic by Hamiltonian vector fields:
\begin{gather*}
H^1(M,\R)\simeq \mathfrak{X}^{\mathrm{sym}}(M)/\mathfrak{X}^{\mathrm{ham}}(M).
\end{gather*}
The map $\pi_*$ sends symplectic vector fields into symplectic vector fields and Hamiltonian into Hamiltonian, and it is surjective on Hamiltonian vector fields.\footnote{This \looseness=-1 is true even when $\pi$ is not the pullback of the inclusion and in full generality follows from the fact that Hamiltonian vector fields are the order zero part of essentially inner derivations, which can always be lifted. For every $f\in C^\infty(S)$, the essentially inner derivation $\frac{1}{\hslash}\left[f\stackrel{\star}{,}\cdot\,\right]=\{f,\,\cdot\,\}+O(\hslash)$ of the star product on $S$ and can be lifted to a derivation $\frac{1}{\hslash}\big[\widetilde{f}\stackrel{\star}{,}\cdot\,\big]=\{\widetilde{f},\,\cdot\,\}+O(\hslash)$ of the star product on $M$ by choosing an extension $\widetilde{f}\in C^\infty(M)$ of $f$. The map $\pi_*$ sends the Hamiltonian vector field $\{\widetilde{f},\,\cdot\,\}$ to $\{f,\,\cdot\,\}$, and it is then surjective on Hamiltonian vector fields.}
If $\pi$ is a coembedding, $\pi_*$ is surjective on symplectic vector fields as well (as we proved at point~(a)). The only thing we have to prove is that every symplectic vector field~$Y$ is in the domain of $\pi_*$, i.e., satisfies~$Y(J)\subset J$ where~$J$ is the vanishing ideal of~$S$. In fact, we are going to prove that this is true for any vector~field.

Let $\omega$ be the symplectic form on~$M$ and $Y\in\mathfrak{X}(M)$. Any $1$-form on $M$ can be written as a~finite sum $\sum\limits_{\mathrm{finite}}f_i{\rm d}g_i$
for some $f_i,g_i\in C^\infty(M)$. Thus
\begin{gather*}
\omega(Y,\,\cdot\,)=\sum_{\mathrm{finite}}f_idg_i=
\sum_{\mathrm{finite}}f_i\omega(X_{g_i},\,\cdot\,)=
\sum_{\mathrm{finite}}\omega(f_iX_{g_i},\,\cdot\,),
\end{gather*}
for some \looseness=-1 $f_i$, $g_i$ and where $X_{g_i}$ denotes the Hamiltonian vector field of $g_i$. By non-degeneracy of~$\omega$:
\begin{gather*}
Y=\sum_{\mathrm{finite}}f_iX_{g_i}.
\end{gather*}
In other words, $\mathfrak{X}^{\mathrm{ham}}(M)$ generates $\mathfrak{X}(M)$ as a $C^\infty(M)$-module.

Now, $X_{g_i}(J)\subset J$ since $J$ is a Lie ideal, and $f_iX_{g_i}(J)\subset J$ since $J$ is an associative ideal, hence the thesis: $Y(J)\subset J$.

(c) is a simple corollary of (b). If $H^1(S,\R)=0$, then up to a factor $1/\hslash$ every derivation of $B$ is inner and $\pi_*$ is surjective on inner derivations. If $H^1(S,\R)\neq 0$, then~\eqref{eq:H1map} cannot be surjective and~\eqref{eq:formalder} cannot be a coembedding.
\end{proof}

Examples of pairs of symplectic manifolds $M,S$ with $S$ (closed embedded) symplectic submanifold of $M$, $H^1(M,\R)=0$ and $H^1(S,\R)\neq 0$ can be constructed as follows. Take $M=\CP^n$ a complex projective space with standard symplectic structure and $S$ any complex submanifold: $\CP^n$ has vanishing cohomology in odd degree, and complex submanifolds are symplectic submanifolds. Now it is not difficult to find examples where $S$ is not simply connected: take any Riemann surface with genus $\geq 1$.

An easier and explicit example of quotient map \eqref{eq:formalder} that is not a coembedding is the following.

\begin{ex}\label{example:28} On $M:=\R^2$ with Cartesian coordinates $(x,y)$ consider the commuting vector fields
\begin{gather*}
X:=\frac{\partial}{\partial x},\qquad
Y:=y\frac{\partial}{\partial y},
\end{gather*}
the Poisson structure given by the bivector field $X\wedge Y$, and the associated Weyl-type star product
\begin{gather}\label{eq:kminstar}
f\star g:=\mu\circ e^{\mathrm{i}\hslash(X\otimes Y-Y\otimes X)}(f\otimes g),\qquad\forall\, f,g\in C^\infty\big(\R^2\big),
\end{gather}
where $\mu$ is the pointwise multiplication map.

Embed $S:=\R$ in $\R^2$ as horizontal axis. Then $\star$ is tangential to $S$ and we have a surjective homomorphism
\begin{gather}\label{eq:commsubman}
\pi\colon \ A:=\big(C^\infty\big(\R^2\big)\h,\star\big)\to B:=C^\infty(\R)\h
\end{gather}
where  the product on the right is the $\C\h$-linear extension of the pointwise product. A vector field
\begin{gather*}
D_0=a\frac{\partial}{\partial x}+b\frac{\partial}{\partial y},\qquad a,b\in C^\infty\big(\R^2\big),
\end{gather*}
is a Poisson vector field on $\R^2$ if and only if
\begin{gather*}
X(a)+Y(b)=0.
\end{gather*}
This implies $\frac{\partial a}{\partial x}|_{y=0}=0$, so $a(x,0)=c$ is constant on the horizontal axis and
\begin{gather*}
\pi_*(D)=c\frac{\partial}{\partial x}+O(\hslash).
\end{gather*}
The (Poisson) vector field $x\frac{\partial}{\partial x}$ on $\R$ is not the restriction of any Poisson vector field on $\R^2$. Using Lemma \ref{lemma:obstr}(c) we conclude that the map \eqref{eq:commsubman} is not a coembedding.
\end{ex}

\begin{rem}
Note that, similarly to Example \ref{ex:no2}, in Example \ref{example:28} one has
\begin{gather*}
x\star y-y\star x=2\mathrm{i}\hslash\, y.
\end{gather*}
Replacing the formal parameter by a non-zero real number,
$\C[x,y]$ with star product \eqref{eq:kminstar}
becomes an algebra isomorphic to the complexification of $\mathcal{U}(\mathfrak{sb}(2,\R))$.
\end{rem}

\subsection{Coisotropic reduction}\label{sec:coiso}
A more general class of examples of surjective homomorphisms of formal deformations comes from coisotropic reduction. Let us review the classical notion of phase space reduction from an algebraic point of view.
Suppose $M$ is a Poisson manifold, that we interpret as a phase space of a physical system, and imagine that the system is constrained to move in a (closed embedded) submanifold $S$ of $M$. To obtain a phase space which represents in some sense the true degrees of freedom of the system, we can then perform phase space reduction.
Recall that $S$ is a \emph{coisotropic} submanifold of $M$ if the vanishing ideal
\begin{gather*}
J:=I(S)=\big\{f\in C^\infty(M)\colon f|_S=0\big\}
\end{gather*}
is closed under the Poisson bracket. Let
\begin{gather*}
A:=\big\{f\in C^\infty(M)\colon  \{f,J\}\subset J \big\}
\end{gather*}
be the \emph{normalizer} of $J$.
Since $J$ is closed under Poisson bracket, $J\subset A$. Since $f\mapsto -X_f:=- \{f,\,\cdot\,\}$ is a Lie algebra morphism, $A$ is closed under Poisson bracket; since $X_{fg}=fX_g+gX_f$ and $J$ is an ideal, $A$ is closed under pointwise product: it is a Poisson subalgebra of $\big(C^\infty(M),\{\,,\,\}\big)$. Finally, by construction $J$ is a Poisson ideal in $A$.\footnote{The normalizer is indeed the largest Poisson subalgebra of $\big(C^\infty(M),\{\,,\,\}\big)$ containing $J$ as a Poisson ideal.}
We thus have an exact sequence \eqref{eq:shortexact} of \emph{Poisson algebra maps}, where the quotient algebra $B$ is interpreted as algebra of smooth functions on what we would call geometrically the reduced phase space $M_{\mathrm{red}}$.\footnote{Note that if $S$ is a symplectic submanifold, then $A=C^\infty(M)$ and $M_{\mathrm{red}}\simeq S$.}

``Good'' star products on $M$ induce formal deformations of the constrained ideal, normalizer and reduced phase space, fitting an exact sequence:
\begin{gather}\label{eq:coisotropic}
0\to (J\h,\star)\to (A\h,\star) \xrightarrow{\pi} \big(C^\infty(M_{\mathrm{red}})\h,\star\big)\to 0.
\end{gather}
A procedure that always works when $S$ has codimension $1$ in $M$ is in \cite{Glo}.
For a general discussion of the problem one can see the review \cite{Bor} and references therein. A more recent ``categorical'' approach is in~\cite{DEW}.

Since symplectic submanifolds are special examples of this construction, we cannot expect~\eqref{eq:coisotropic} to be always a coembedding.

As a concrete example one can take $M:=\C^{n+1}\setminus\{0\}$ with standard symplectic structure. The submanifold $S:=\mathbb{S}^{2n+1}$ is then coisotropic. Let $\partial_\theta$ be the vector field on $M$ (tangent to~$S$) generating the obvious $U(1)$ action given by multiplication of all complex coordinates by the same phase. The normalizer $J$ is given by those functions $f\in C^\infty(M)$ such that $\partial_\theta f$ vanishes on~$S$. The quotient algebra is isomorphic to the algebra of smooth functions on $\mathbb{S}^{2n+1}$ that are $U(1)$-invariant, that we identify with smooth functions on $M_{\mathrm{red}}=\CP^n$.
It is shown in \cite{Glo} that Wick star product on $M$ can be reduced to $\CP^n$ and one has a sequence \eqref{eq:coisotropic} of formal deformations. In this example $\pi_*$ is surjective again for trivial reasons: $M_{\mathrm{red}}$ is simply connected and all derivations are essentially inner.

One could enlarge the class of examples by looking at reductions coming from an action of a~Poisson--Lie group on a Poisson manifold~\cite{DNE}, or at coisotropic submanifolds of Jacobi manifolds (rather than Poisson)~\cite{LOTV}. It is also worth mentioning the paper \cite{CF}, where the authors prove a relative version of Kontsevich's formality theorem, involving formal deformations of pairs of a manifold and a submanifold (covering in particular the case of Poisson manifold together with a coisotropic submanifold). It could be interesting to see if there is any connection with submanifold algebras. This investigation is postponed to a future work.

\subsection*{Acknowledgements}
I would like to thank Alessandro De Paris for suggesting Example~\ref{ex:cross2} and Chiara Esposito for her comments on a preliminary version of the paper. A special thanks goes to the anonymous referees for carefully reading the paper and suggesting some interesting future research lines.

\pdfbookmark[1]{References}{ref}
\LastPageEnding

\end{document}